\documentclass[12pt,a4paper,oneside]{amsart}

\usepackage{amsfonts, amsmath, amssymb, amsthm, hyperref}
\usepackage{anysize}
\usepackage{color}
\usepackage{graphics}

\newtheorem{theorem}{Theorem}[section]
\newtheorem*{theorem*}{Theorem}

\newtheorem{corollary}[theorem]{Corollary}
\newtheorem{problem}[theorem]{Problem}

\theoremstyle{definition}
\newtheorem{definition}[theorem]{Definition}

\theoremstyle{remark}
\newtheorem{remark}[theorem]{Remark}

\numberwithin{equation}{section}

\renewcommand{\epsilon}{\varepsilon}
\renewcommand{\phi}{\varphi}
\renewcommand{\kappa}{\varkappa}

\newcounter{fig}
\newcommand{\f}{\refstepcounter{fig} Fig. \arabic{fig}. }

\begin{document}

\title{Cutting the same fraction of several measures}

\author{Arseniy~Akopyan}

\address{Arseniy Akopyan, Institute for Information Transmission Problems RAS, Bolshoy Karetny per. 19, Moscow, Russia 127994
\newline \indent
B.N.~Delone International Laboratory ``Discrete and Computational Geometry'', Yaroslavl' State University, Sovetskaya st. 14, Yaroslavl', Russia 150000}

\email{akopjan@gmail.com}
\thanks{The research of A.V.~Akopyan is supported by the Dynasty Foundation, the President's of Russian Federation grant MD-352.2012.1, the Russian Foundation for Basic Research grants 10-01-00096 and 11-01-00735, and the Russian government project 11.G34.31.0053.}

\author{Roman~Karasev}

\address{Roman Karasev, Dept. of Mathematics, Moscow Institute of Physics and Technology, Institutskiy per. 9, Dolgoprudny, Russia 141700
\newline \indent
B.N.~Delone International Laboratory ``Discrete and Computational Geometry'', Yaroslavl' State University, Sovetskaya st. 14, Yaroslavl', Russia 150000}

\email{r\_n\_karasev@mail.ru}
\urladdr{http://www.rkarasev.ru/en/}
\thanks{The research of R.N.~Karasev is supported by the Dynasty Foundation, the President's of Russian Federation grant MD-352.2012.1, the Russian Foundation for Basic Research grants 10-01-00096 and 10-01-00139, the Federal Program ``Scientific and scientific-pedagogical staff of innovative Russia'' 2009--2013, and the Russian government project 11.G34.31.0053.}


\begin{abstract}
We study some measure partition problems: Cut the same positive fraction of $d+1$ measures in $\mathbb R^d$ with a hyperplane or find a convex subset of $\mathbb R^d$ on which $d+1$ given measures have the same prescribed value. For both problems positive answers are given under some additional assumptions.
\end{abstract}

\maketitle

\section{Introduction}

The famous ``ham sandwich'' theorem of Stone, Tukey, and Steinhaus~\cite{stone1942generalized,steinhaus1945division} asserts that every $d$ absolutely continuous probability measures in $\mathbb R^d$ can be simultaneously partitioned into equal parts by a single hyperplane. 

In~\cite{bereg2012balanced} M.~Kano and S.~Bereg raised the following question (in the planar case): If we are given $d+1$ measures in $\mathbb R^d$ and want to cut the same (but unknown) fraction of every measure by a hyperplane then what assumptions on the measures allow us to do so? Certainly, additional assumptions are required because if the measures are concentrated near vertices of a $d$-simplex then such a fraction cut is impossible. A sufficient assumption is described below:

\begin{definition}
Let $\mu_0$, $\mu_1$, \dots, $\mu_d$ be absolutely continuous probability measures on $\mathbb R^d$ and let $\varepsilon\in (0, 1/2)$. Call the set of measures \emph{$\varepsilon$-not-permuted} if for any halfspace $H$ the inequalities $\mu_i(H) < \varepsilon$ for all $i=0, 1, \dots, d$ imply
\[
\mu_i(H)\ge \mu_j(H), \, \text{for some } i<j.
\]
\end{definition}

\begin{remark}
\label{re:not-permutation}
For $\varepsilon>0$ consider all halfspaces $H$ in $\mathbb{R}^d$ such that $\mu_i(H)<\varepsilon$ for all $i$ and the values $\mu_i(H)$ are pairwise distinct. If we arrange the values $\mu_i(H)$ in the ascending order then we get some permutation of $\{0, 1, \dots, d\}$. So the measures $\mu_i$ are $\varepsilon$-not-permuted if and only if in such a way we cannot get all possible permutations of the $d+1$ element set.
\end{remark}

\begin{remark}
A natural example of $\varepsilon$-not-permuted measures appears when the support of one measure lies in the interior of the convex hull of the union of supports of the other $d$ measures. In this case the measures are $\varepsilon$-not-permuted for sufficiently small $\varepsilon$.
\end{remark}

Now we state the main result:

\begin{theorem}
\label{thm:fraction-cut}
Suppose $\mu_0$, $\mu_1$, $\dots$, $\mu_d$ are absolutely continuous probability $\varepsilon$-not-permuted measures in $\mathbb R^d$ for some $\varepsilon\in(0, 1/2)$. Then there exists a halfspace $H$ such that
\[
\mu_0(H) = \mu_1(H) = \dots = \mu_d(H) \in [\varepsilon, 1/2].
\]
\end{theorem}

Note that in~\cite{barany2008slicing,breuer2010uneven,karasev2009theorems} a similar problem was considered: Cut by a hyperplane a \emph{prescribed} fraction of each of $d$ measures in $\mathbb R^d$. Again, this cannot be done in general and some additional assumptions were needed.

A straightforward consequence of Theorem~\ref{thm:fraction-cut} follows by considering one measure concentrated near a point:

\begin{corollary}
Suppose $\mu_1$, $\dots$, $\mu_d$ are absolutely continuous probability measures in $\mathbb R^d$ and $p$ is a point in the convex hull of their supports. Then there exists a halfspace $H$ such that
\[
\mu_1(H) = \mu_2(H) = \dots = \mu_d(H)
\]
and $p\in \partial H$.
\end{corollary}

In Section~\ref{discrete-sec} we consider a discrete version of Theorem~\ref{thm:fraction-cut}, replacing measures by finite point sets. This is in accordance with the initial statement of the problem in~\cite{bereg2012balanced}.

Finally, in Section~\ref{convex-cut-sec} we consider a problem of cutting the same \emph{prescribed} fraction of every measure, this time allowing cutting with a convex subset of $\mathbb R^d$.

\medskip
{\bf Acknowledgments.} 
We thank Nikolay~Dolbilin for drawing our attention to this problem. We also thank Fedor Petrov, Imre B\'ar\'any, and Pavle Blagojevi\'c for useful remarks.

\section{Ham sandwich theorem for charges}

In order to prove Theorem~\ref{thm:fraction-cut} we need a version of the ``ham sandwich'' theorem~\cite{stone1942generalized,steinhaus1945division} for charges: 

\begin{definition}
A difference $\rho = \mu' - \mu''$ of two absolutely continuous finite measures on $\mathbb R^d$ is called a \emph{charge}. In other words, the charge is represented by its density from $L_1(\mathbb R^d)$. 
\end{definition}

\begin{theorem}[Ham sandwich for charges]
\label{thm:ham-sandwich}
Suppose we are given $d$ charges $\rho_1$, $\dots$, $\rho_d$ in $\mathbb R^d$, then there exists a (possibly degenerate) halfspace $H$ such that for any $i$
\[
\rho_i(H) = 1/2 \rho_i(\mathbb R^d).
\]
\end{theorem}

\begin{remark}
A \emph{degenerate halfspace} is either $\emptyset$ of the whole $\mathbb R^d$. We cannot exclude degenerate halfspaces in this theorem when $\rho_i(\mathbb R^d)$'s are all zero. The reason is the same for which we cannot exclude the ``$\varepsilon$-not-permuted'' assumption from Theorem~\ref{thm:fraction-cut}, see also Remark~\ref{counterexample} below.
\end{remark}

\begin{proof}
The classical proof from the book of Matou\v{s}ek~\cite{matousek2003using} passes for charges (the authors learned this fact long ago from Vladimir~Dol'nikov). Identify $\mathbb R^d$ with $\mathbb R^d\times \{1\} \subset \mathbb R^{d+1}$. Parameterize halfspaces $\tilde H\subset \mathbb R^{d+1}$ with boundary passing through the origin by their inner normals. So all halfspaces $H = \tilde H\cap \mathbb R^d\times \{1\}$ (including degenerate) are parameterized by the unit sphere $S^d$.

If we map every halfspace to the vector 
\[
(\rho_1(H) - \rho_1(\mathbb R^d\setminus H), \dots, \rho_d(H) - \rho_d(\mathbb R^d\setminus H))
\]
then we obtain a continuous odd map $P : S^d\to \mathbb R^d$; by the Borsuk--Ulam theorem~\cite{borsuk1933drei} (see also~\cite{matousek2003using}) one halfspace must be mapped to zero.
\end{proof}

\section{Proof of Theorem~\ref{thm:fraction-cut}}

As the first attempt we try to apply the ham sandwich theorem for charges to 
\[
\rho_1 = \mu_1 - \mu_0,\, \rho_2 = \mu_2 - \mu_1,\, \dots,\, \rho_d = \mu_d - \mu_{d-1}.
\]
This way we easily obtain a halfspace $H$ such that $\mu_0(H) = \mu_1(H) = \dots = \mu_d(H)$. But the halfspace $H$ may be degenerate or $\mu_i(H)$ may be all zero. This is not what we want.

\begin{remark}
\label{counterexample}
By the way, we see that starting from three measures $\mu_i$ on $\mathbb R^2$ distributed along three rays emanating from vertices of a regular triangle and going outside it, we cannot cut the same fraction (possibly zero) of these measures by a non-degenerate halfplane. The same example generalizes to higher dimensions and shows that in the ham sandwich theorem for charges we cannot avoid using degenerate halfspaces when all charges satisfy $\rho_i(\mathbb R^d) = 0$.
\end{remark}

So let us perturb the charges with a small positive parameter $s$:
\[
\rho^s_i = (1+s)\mu_i - \mu_{i-1},\, \text{for }i=1, \dots, d.
\]
Now Theorem~\ref{thm:ham-sandwich} gives a halfspace $H$ with:
\begin{equation}
\label{eq:equal-part}
\rho^s_1(H) = \dots = \rho^s_d(H)=s/2.
\end{equation}
This is equivalent to the following equalities:
\begin{equation}
\label{eq:s-part}
\mu_{i-1}(H) = (1+s)\mu_i(H) - s/2 = \mu_i(H) + s (\mu_i(H) - 1/2).
\end{equation}
Suppose  $\mu_i(H) < \varepsilon$ for all $i=0, \dots, d$.
Then we have $\mu_i(H) < 1/2$ and $s (\mu_i(H) - 1/2)<0$. Hence
\begin{equation}
\label{eq:ths-ineq}
\mu_{i-1}(H)<\mu_i(H),\, \text{for }i=1, \dots, d.
\end{equation}

This contradicts the assumption that the measures are \emph{$\varepsilon$-not-permuted}.

Hence the inequality $\mu_i(H)\ge \varepsilon$ for some $i$ is guaranteed while we decrease $s$ to $0$, and in turn it guarantees (remember that we can interchange $H$ and $\mathbb R^d\setminus H$!) that $H$ cannot approach degenerate halfspaces $\emptyset$ and $\mathbb R^d$. So we assume by compactness that $H$ tends to a certain halfspace as $s\to 0$ and going to the limit in (\ref{eq:s-part}) together with (\ref{eq:equal-part}) we obtain the conclusion.

\section{Discrete version}
\label{discrete-sec}

In the paper \cite{bereg2012balanced} Bereg and Kano consider a discrete version of this theorem in the plane. They call a line $\ell$ \emph{balanced} if each half-plane bounded by $\ell$ contains precisely the same number of points of each color.
\begin{theorem}[S.~Bereg and M.~Kano, \cite{bereg2012balanced}]
	\label{thm:bereg kano}
	Let $S$ be a set of $3n \geq 6$ points in the plane in general position colored in red/blue/green such that\\
	(i) the number of points of each color is $n$;\\
	(ii) the vertices of the convex hull of $S$ have the same color.\\
	Then there exists a balanced line of $S$.
\end{theorem}

Here we generalize the result of Bereg and Kano as follows:

\begin{theorem}
	\label{thm:fraction-cut discrete}
	Let $S$ be a set of $(d+1)n$ points in $\mathbb{R}^d$ in general position colored in colors $0, 1, \dots, d$ so that:\\
	(i) the number of points of each color is $n$;\\
	(ii) for any directed line $\ell$ there exist two colors $i$ and $j$, $i<j$ such that the farthest in the direction of $\ell$ point of color $i$ is not closer (in the direction of $\ell$) than any point of color~$j$. \\
	Then there exists a \emph{balanced} hyperplane $h$, in other words, the hyperplane such that each half-space bounded by $h$ contains precisely the same number of points of each color.
\end{theorem}
\begin{proof}
	The proof follows almost directly from the continuous version. Replace each point by a ball solid ball centered in it and of radius $r>0$ sufficiently small so that there is no hyperplane that intersects any $d+1$ balls at the same time and for any partition of points of $s$ by a hyperplane there exist a hyperplane that separates the corresponding balls accordingly. Since the points are in general position such $r$ does exist.
	
	Balls of each color generate an absolutely continuous measure. So we have $d+1$ absolutely continuous measures; multiplying them by the same constant we make these measures probabilistic, denote them by $\mu_0$, $\mu_1$, \dots, $\mu_d$.
	
	Now as in the proof of Theorem \ref{thm:fraction-cut} we consider charges $\rho^s_i = (1+s)\mu_i - \mu_{i-1}$ and find a halfspace $H$ such that $\rho^s_1(H) = \dots = \rho^s_d(H)=s/2.$ 

We will show that $\mu_i(H)$ could not be less than $1/n$ for all $i$. Suppose it is so. Then again using equation \eqref{eq:s-part} we get:
	\begin{equation}
		\label{eq:ths-ineq second}
	\mu_{i-1}(H)<\mu_i(H),\, \text{for }i=1, \dots, d.
	\end{equation}

Consider a line $\ell$ inner normal to $H$ and colors $i$ and $j$ from condition (ii) for this line. Suppose $\mu_{i}(H)<\mu_j(H)<1/n$. The hyperplane $\partial H$ intersects at least two balls of color $j$, otherwise the inequality would be $\mu_{i}(H)\ge \mu_j(H)$ by condition (ii). Under the above assumptions, every hyperplane can pass through at most $d$ balls therefore there exist two colors $k$ and $m$ that do not intersect with $\partial H$ and therefore with $H$ (in the opposite case $\mu_k(H)$ or $\mu_m(H)$ would be at least $1/n$). Hence we have $\mu_k(H) = \mu_m(H) = 0$ in contradiction with \eqref{eq:ths-ineq second}. So we for at least one $i$ the measure $\mu_i(H)$ is at least $1/n$.

As $s$ goes to $0$, the halfspace $H$ tends to a certain halfspace $H_0$. Its border hyperplane $h$ cuts equal positive fraction of every measure $\mu_i$.

If $h$ touches no ball where the measures are concentrated then we are done. But a problem can occur if $h$ intersects some balls. Since the points in $S$ are in general position, the hyperplane $h$ touches at most $d$ of them (denote the corresponding subset of $X$ by $I$) and we can perturb $h$ so that arbitrary subset $J\subseteq I$ will be on one side of $h$ while $i\setminus J$ will be on the other side of $h$. Thus we may ``round'' the fractions of the measure in any way we want and equalize the numbers of points in $H$ for all the colors (compare with the proof of Corollary 3.1.3 in \cite{matousek2003using}).
\end{proof}

\begin{remark}
	\label{re:for fraction discrete}
	As in remark \ref{re:not-permutation} we can describe point sets satisfying condition (ii) in terms of permutations: For any directed line $\ell$ the order of points with least (among the point of the same color) projection to $\ell$ gives a permutation of colors $\{0, 1, \dots, d\}$; a colored point set satisfies condition (ii) if and only if we cannot get all permutations this way. 
\end{remark}

Let us give a simpler but still powerful sufficient assumption on $S$:

\begin{corollary}
	\label{cor:points-convex}
	Let $S$ be a set of $(d+1)n$ points in $\mathbb{R}^d$ in general position colored in colors $0, 1, \dots, d$ so that:\\
	(i) the number of points of each color is $n$;\\
	(ii) points of one color lie in the convex hull of the union of points of other $d$ colors.\\
	Then there exists a \emph{balanced} hyperplane $h$.
\end{corollary}

\section{Cutting a prescribed fraction by a convex set}
\label{convex-cut-sec}

Let us state another problem about cutting the same fraction of several measures:

\begin{problem}
\label{prob:cutting-alpha}
The dimension $d$ and the number of measures $k>1$ are given. For which $\alpha>0$ for any  absolutely continuous probability measures $\mu_1$, $\dots$, $\mu_k$ on $\mathbb R^d$ it is always possible to find a convex subset $C\subset \mathbb R^d$ such that
\[
\mu_1(C) = \dots = \mu_k(C)=\alpha?
\]
\end{problem}

For $\alpha=1/2$ and $k=d$ a positive solution to this problem follows from the ham sandwich theorem.  If $\alpha>1/2$ then considering two measures, one concentrated near the origin and the other concentrated uniformly near a unit sphere, we see that there is no solution. 

If $k>d+1$ one can consider $d+2$ measures concentrated near the vertices of a simplex $S$ and the mass center $w$ of $S$; in this case any such $C$ must contain a neighborhood of $w$ and therefore cannot cut $\alpha$ of the corresponding measure.

In~\cite{stromquist1985sets} Stromquist and Woodall solved a similar problem for $k$ measures on a circle and cutting it by a union of $k$ arcs.

Using the results on measure equipartitions from~\cite{soberon2010balanced,aronov2010convex,karasev2010equipartition} we are able to solve Problem \ref{prob:cutting-alpha} for $d+1$ measures:

\begin{theorem}
\label{cutting-one-kth}
Suppose $\mu_0$, $\mu_1$, $\dots$, $\mu_d$ are absolutely continuous probability measures on $\mathbb R^d$ and $\alpha \in (0, 1)$. It is always possible to find a convex subset $C\subset \mathbb R^d$ such that
\[
\mu_0(C)=\mu_1(C) = \dots = \mu_d(C)=\alpha,
\]
if and only if $\alpha = 1/m$ for a positive integer $m$.
\end{theorem}

\begin{proof}
First, we prove that it is possible if $\alpha=1/m$.
Following~\cite{soberon2010balanced} it is sufficient to consider the case when $m$ is a prime and use induction. In this case we are going to apply~\cite[Theorem~1.3]{karasev2010equipartition} to these measures and the space of functions 
\[
L = \left\{a_0 + \sum_{i=1}^d a_ix_i + b \sum_{i=1}^d x_i^2\right\}.
\]
This space has dimension $d+2$, which is sufficient to partition $d+1$ measures into equal parts by a generalized Voronoi partition. Recall that a generalized Voronoi partition corresponding to an $m$-tuple of pairwise distinct functions $\{f_1,\dots, f_m\} \subset L$ is defined by 
\[
C_i = \{x\in \mathbb R^d : f_i(x) \le f_j(x)\ \text{for all}\ j\neq i\}.
\]
Assume without loss of generality that $f_1(x)$ has the largest coefficient at $\sum_{i=1}^d x_i^2$ among all $f_j(x)$. Then the defining equations for $C_1$ will look like 
\[
(b_1-b_j)\sum_{i=1}^d x_i^2 + \lambda(x) \le 0,
\]
where $\lambda(x)$ is a linear function and $b_1-b_j$ is nonnegative. Note that each of these equations defines either a halfspace or a ball and therefore their intersection $C_1$ is convex.

%
%

Now we give a counterexample for $d=1$ and $\alpha$ not of the form $1/m$. Assume $\frac{1}{n}>\alpha>\frac{1}{n+1}$. Let $\mu_0$ to be the uniform measure on $(0,1)$.

Let $a_i$, $i=1, \dots, n$ be the points with coordinates $\frac{i}{n+1}$ and $\Delta_i$ be the intervals with centers at $a_i$ an length $\epsilon<\alpha-\frac{1}{n+1}$. The support of the measure $\mu_1$ is the union of intervals $\Delta_i$ and $\int_{\Delta_i}d\mu_1=\frac{1}{n}$ for each $i$ (Fig.~\ref{fig:one-dimension}).

It easy to see that each convex set $C$ with $\mu_0(C)=\alpha$ is an interval of length $\alpha$ in intersection with $(0, 1)$ and it must contain at least one interval $\Delta_i$. Therefore $\mu_1(C)\geq\frac{1}{n}>\alpha$.

For $d>1$ we can extend the one-dimensional example. Consider a $d$-dimensional regular simplex with vertices $v_0$, $v_1$, \dots, $v_d$. Let the measures $\mu_2$, \dots, $\mu_d$ concentrate near the vertices $v_2$, \dots, $v_d$ respectively. Like we did it for $d=1$ let $\mu_0$ be the uniform measure on the tiny cylinder around the edge $v_0v_1$. The measure $\mu_1$ will look like in the one-dimensional case, but its support will be intervals on the segment that connects centers of faces $v_0v_2\dots v_d$ and $v_1v_2\dots v_d$. See Fig.~\ref{fig:two-dimension} for the two-dimensional case. 
	
	\begin{center}
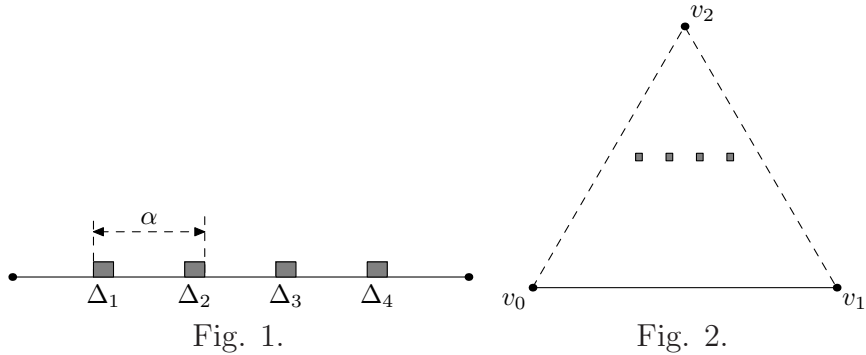
		\begin{center}
		\begin{tabular}{cc}
			\includegraphics{figcutting-1.mps} &
			\includegraphics{figcutting-2.mps}	\\
			\f \label{fig:one-dimension} &
			\f \label{fig:two-dimension}\\
		\end{tabular}
		\end{center}
	\end{center}

To conclude the proof we note the following. Assume that the measures $\mu_i$ are concentrated in $\delta$-neighborhoods of their respective vertices (for $i>1$) or segments (for $i=0,1$). Then going to the limit $\delta\to+0$ and using the Blaschke selection theorem we assume that the corresponding $C_\delta$ tend in the Hausdorff metric to some $C_0$. This $C_0$ must intersect the segment $[v_0,v_1]$ by an interval of length at least $\alpha |v_0-v_1|$. It also has to contain every vertex $v_i$ for $i=2,\dots, d$. Hence, similar to the one-dimensional case, $C_0$ contains in its interior one of the support segments of $\mu_1$. For small enough $\delta$, the convex set $C_\delta$ will also contain at least $1/n$ of the measure $\mu_1$, which is a contradiction.

\end{proof}

\begin{remark}
\label{re:levy-hopf}
For $d=1$ this result follows from the theorem of Levy~\cite{levy1934surune} about a segment on a curve. Note that Hopf~\cite{hopf1936uber} showed that the set of $\alpha$ for which the required segment does not exist is additive.
\end{remark}

\begin{remark}
In~\cite[Theorem~3.2]{blagojevic2007using} it was proved that any two absolutely continuous probability measures on $S^2$ can be cut into pieces of measures $\alpha, \alpha, 1-2\alpha$ by a $3$-fan. Using the central projection we see that any two absolutely continuous probability measures on $\mathbb R^2$ may be cut into pieces of measures $\alpha, \alpha, 1-2\alpha$ with a (possibly degenerate) $3$-fan. It is clear that at least one of the $\alpha$ parts of the fan is a convex angle and therefore Problem~\ref{prob:cutting-alpha} has a positive solution for $k=d=2$ and any $\alpha\in (0,1/2]$.
\end{remark}


\begin{thebibliography}{10}

\bibitem{aronov2010convex}
B.~Aronov and A.~Hubard.
\newblock Convex equipartitions of volume and surface area.
\newblock {\em Arxiv preprint arXiv:1010.4611}, 2010.

\bibitem{barany2008slicing}
I.~B{\'a}r{\'a}ny, A.~Hubard, and J.~Jer\'{o}nimo.
\newblock Slicing convex sets and measures by a hyperplane.
\newblock {\em Discrete \& Computational Geometry}, 39:67--75, 2008.
\newblock 10.1007/s00454-007-9021-2.

\bibitem{bereg2012balanced}
S.~Bereg and M.~Kano.
\newblock Balanced line for a 3-colored point set in the plane.
\newblock {\em The Electronic Journal of Combinatorics}, 19(1):P33, 2012.

\bibitem{blagojevic2007using}
P.~V.~M. Blagojevi\'c and A.~S.~D. Blagojevi\'c.
\newblock {Using equivariant obstruction theory in combinatorial geometry.}
\newblock {\em Topology and its Applications}, 154(14):2635--2655, 2007.

\bibitem{borsuk1933drei}
K.~Borsuk.
\newblock {D}rei {S}{\"a}tze {\"u}ber die $n$-dimensionale euklidische
  {S}ph{\"a}re.
\newblock {\em Fundam. Math}, 20(1):177--190, 1933.

\bibitem{breuer2010uneven}
F.~Breuer.
\newblock Uneven splitting of ham sandwiches.
\newblock {\em Discrete \& Computational Geometry}, 43(4):876--892, 2010.

\bibitem{hopf1936uber}
H.~Hopf.
\newblock \"{U}ber die {S}ehnen ebener {K}ontinuen und die {S}chleifen
  geschlossener {W}ege.
\newblock {\em Commentarii Mathematici Helvetici}, 9:303--319, 1936.
\newblock 10.1007/BF01258195.

\bibitem{karasev2009theorems}
R.~N. Karasev.
\newblock {Theorems of {B}orsuk-{U}lam type for flats and common transversals
  of families of convex compact sets}.
\newblock {\em Sb. Math.}, 200(10):1453--1471, 2009.

\bibitem{karasev2010equipartition}
R.~N. Karasev.
\newblock Equipartition of several measures.
\newblock {\em Arxiv preprint arXiv:1011.4762}, 2010.

\bibitem{levy1934surune}
P.~Levy.
\newblock {Sur une g\'en\'eralisation du th\'eor\`eme de {R}olle.}
\newblock {\em C. R. Acad. Sci., Paris}, 198:424--425, 1934.

\bibitem{matousek2003using}
J.~Matou{\v{s}}ek.
\newblock {\em Using the Borsuk-Ulam theorem: lectures on topological methods
  in combinatorics and geometry}.
\newblock Springer Verlag, 2003.

\bibitem{soberon2010balanced}
P.~Sober\'on.
\newblock {Balanced convex partitions of measures in $\mathbb R^{d}$.}
\newblock {\em Mathematika}, 58(1):71--76, 2012.

\bibitem{steinhaus1945division}
H.~Steinhaus.
\newblock Sur la division des ensembles de l’espace par les plans et des
  ensembles plans par les cercles.
\newblock {\em Fundam. Math}, 33:245--263, 1945.

\bibitem{stone1942generalized}
A.~H. Stone and J.~W. Tukey.
\newblock Generalized ``sandwich'' theorems.
\newblock {\em Duke Mathematical Journal}, 9(2):356--359, 1942.

\bibitem{stromquist1985sets}
W.~Stromquist and D.~Woodall.
\newblock {Sets on which several measures agree.}
\newblock {\em J. Math. Anal. Appl.}, 108:241--248, 1985.

\end{thebibliography}
\end{document}